\documentclass{amsart}

\usepackage[inactive]{srcltx} 
\usepackage{amsmath, amsthm, amscd, amsfonts, amssymb, graphicx, color, extarrows}
 \newtheorem{thm}{Theorem}[section]
 \newtheorem{cor}[thm]{Corollary}
 \newtheorem{lem}[thm]{Lemma}

 \theoremstyle{definition}
 
 \theoremstyle{remark}
 
 \newtheorem{exam}[thm]{Example}
 \numberwithin{equation}{section}
\newcommand{\T}{ \textbf{t}_{\lambda}}
\newcommand{\PT}{\mathfrak{U}\widehat{\otimes} \mathfrak{U}}

\newcommand{\uu}{\mathfrak{U}}
\begin{document}

\title[Jordan and Lie derivations of $\phi  $-Johnson amenable Banach algebras]
 {Jordan and Lie derivations of $\phi  $-Johnson amenable Banach algebras}

\author{Hoger Ghahramani and Parvin Zamani}
\thanks{{\scriptsize
\hskip -0.4 true cm \emph{MSC(2020)}: 47B47; 46H20; 46H99; 43A07. 
}}

\address{Department of Mathematics, Faculty of Science, University of Kurdistan, P.O. Box 416, Sanandaj, Kurdistan, Iran.}
\email{h.ghahramani@uok.ac.ir; hoger.ghahramani@yahoo.com}

\address{Department of Mathematics, Faculty of Science, University of Kurdistan, P.O. Box 416, Sanandaj, Kurdistan, Iran.}
\email{parvin.zamani88@gmail.com}


\address{}

\email{}

\thanks{}

\thanks{}



\date{}

\dedicatory{}

\commby{}


\begin{abstract}
Let $\uu$ be a $\phi $-Johnson amenable Banach algebra in which $\phi$ is a non-zero multiplicative linear functionals on $\uu$. Suppose that $X$ is a Banach $\uu$-bimodule such that $a.x=\phi(a)x$ for all $a\in \uu$, $x\in X$ or $x.a=\phi(a)x$ for all $a\in \uu$, $x\in X$. We show that every continuous Jordan derivation from $\uu$ to $X$ is a derivation, and every continuous Lie derivation from $\uu$ to $X$ decomposed into the sum of a continuous derivation and a continuous center-valued trace. Then we apply our results for character amenable Banach algebras and amenable Banach algebras. We  also provide some results about $\phi$-Johnson amenability, especially we give some conditions equivalent to $\phi $-Johnson amenability.
\end{abstract}

\maketitle
\section{Introduction }
The study of Jordan and Lie structures of associative algebras is one of the favorite study paths in mathematics. Among the interesting cases in this field of study, we can refer to the study of Jordan derivations and Lie derivations on associative algebras. Let $\uu$ be an associative algebra, $X$ be a $\uu$-bimodule, and $\mathcal{Z}_{\uu}(X)$ be the centre of $X$. We set $a\circ b= ab+ba$ (Jordan product) and $[a,b]=ab-ba$ (Lie product) for all $a,b \in \uu$. Recall that a linear mapping $\delta : \uu \rightarrow X$  is called a \textit{derivation} if 
\[ \delta(ab)=\delta(a)b+a\delta(b) \quad\quad (a,b\in \uu),\]
a \textit{Jordan derivation} if 
\[ \delta(a\circ b)=\delta(a)\circ b+a\circ\delta(b) \quad\quad (a,b\in \uu),\]
and a \textit{Lie derivation} if 
\[ \delta([a, b])=[\delta(a), b]+[a,\delta(b)] \quad\quad (a,b\in \uu).\]
Moreover, if $d: \uu \rightarrow X$ is a derivation and $\varphi: \uu\rightarrow \mathcal{Z}_{\uu}(X)$ is a \textit{centre-valued trace}, that is, a linear mapping such that $\varphi([a,b])=0$ for all $a,b\in \uu$, then the map $\delta=d+\varphi$ is a Lie derivation. We shall say that is a Lie derivation $\delta$ is in \textit{standard form} in the case $\delta=d+\varphi$ for $d$ and $\varphi$ as above. Jordan derivations and Lie derivations of associative algebras play significant roles in various mathematical areas, in particular in matrix theory, in ring theory, in operator algebra theory and in the theory of Banach algebras. There are two common problems in this context: 
\begin{itemize}
\item[(1)] whether Jordan derivations are derivations; 
\item[(2)] whether Lie derivations are of the standard form. 
\end{itemize}
\par
Many studies have been carried out in line with the two problems raised, and here we mention some of them that are related to Banach algebras. 
In relation to Problem 1, first Herstein \cite{her} proved that every additive Jordan derivation on a 2-torsion-free prime ring is an additive derivation. Sinclair \cite{sinc} showed that every continuous Jordan derivation on a semisimple Banach algebra is a derivation, and later in \cite{cus}, Herstein’s result was generalized for 2-torsion-free semiprime rings. Johnson showed in \cite{joh3} that any continuous Jordan derivation from a symmetrically amenable Banach algebra $\uu$ into a Banach $\uu$-bimodule is a derivation. As a consequence he obtained the same result for continuous Jordan derivations on arbitrary $C^{*}$-algebras, although not every (amenable) $C^{*}$-algebra is symmetrically amenable. Then in \cite{per} it was proved that every Jordan derivation from a $C^{*}$-algebra $\uu$ into a Banach $\uu$-bimodule $X$ is a derivation. For more studies concerning Jordan derivations we refer the reader to \cite{alaminos0, gh1, gh2} and the references therein. Regarding Problem 2, which is actually along Herstein’s programme \cite{her2}, Martindale \cite{mar} proved that an additive Lie derivation of certain primitive ring is always a sum of a derivation and an additive mapping from the ring into its centroid. Miers \cite{mie} considered von Neumann algebras and showed that every Lie derivation is a standard Lie derivation. Johnson showed in \cite{joh3} that any continuous Lie derivation from a symmetrically amenable Banach algebra $\uu$ into a Banach $\uu$-bimodule decomposed into the sum of a continuous derivation and a continuous center-valued trace, and as a consequence he obtained the same result for continuous Lie derivations on arbitrary $C^{*}$-algebras. In \cite{math}, it was proved that every Lie derivation on a $C^{*}$-algebra into itself is in standard form. Then in \cite{alaminos0} it was showed that all Lie derivations from a von Neumann algebra $\uu$ into any Banach $\uu$-bimodule are standard. Continuing the proof of Johnson's results without the continuity condition, the authors showed in \cite{alaminos} that every Lie derivation on a symmetrically amenable semisimple Banach algebra into itself can be uniquely decomposed into the sum of a derivation and a centre-valued trace. To study more about the obtained results concerning determination of the structure
of Lie derivations see \cite{alaminos0, beh, gh0, li} and references therein.  
\par 
We know that every symmetric amenable Banach algebra is an amenable Banach algebra, and the converse of this is not necessarily true (see \cite{joh3}). On the way to the answer to Problems 1 and 2 and according to the mentioned results, the question arises whether it is possible to obtain Johnson's results for Jordan derivations and Lie derivations from amenable Banach algebras into appropriate Banach modules? In this paper, we provide the answer to this question for a larger class of amenable Banach algebras, called $\phi $-Johnson amenable Banach algebras. More precisely, if $\uu$ is a $\phi $-Johnson amenable Banach algebra in which $\phi$ is a non-zero multiplicative linear functionals on $\uu$, and $X$ is a Banach $\uu$-bimodule such that $a.x=\phi(a)x$ for all $a\in \uu$, $x\in X$ or $x.a=\phi(a)x$ for all $a\in \uu$, $x\in X$, then we show that every continuous Jordan derivation from $\uu$ to $X$ is a derivation, and every continuous Lie derivation from $\uu$ to $X$ decomposed into the sum of a continuous derivation and a continuous center-valued trace (Theorems \ref{jordan} and \ref{Lie}). To prove these results, we first present some conditions equivalent to $\phi $-Johnson amenability, which can be interesting in themselves (Theorem \ref{equivalent phi amenable}). As an application, we apply the main results for character amenable Banach algebras (Corollaries \ref{Jordan character amenable} and \ref{Lie character amenable}) and amenable Banach algebras (Corollaries \ref{Jordan amenable} and \ref{Lie amenable}), and also obtain a necessary condition for $ \phi$-Johnson amenability, character amenability and amenability based on point Jordan derivations (Corollary \ref{point Jordan der}). Further, several examples are presented which illustrate limitations on extending some of the theory developed (Examples \ref{example jordan der} and \ref{example Lie der}). This article is organized as follows. In Section 2, definitions and required results are introduced. Section 3 is devoted to some results about $ \phi$-Johnson amenability, especially the equivalent conditions with $ \phi$-Johnson amenability are presented in this section. Section 4 is about Jordan derivations of $ \phi$-Johnson amenable Banach algebras and related results. In Section 5, Lie derivations of $ \phi$-Johnson amenable Banach algebras and related results are studied.
\section{Preliminaries }
In this section we fix the notation, and recall some basic definitions and points which will be used in the next sections. Assume that $\uu$ is a Banach algebra and $X$ is a Banach $\uu$-bimodule. Note that $\mathcal{Z}_{\uu}(X)$ represents the \textit{centre} of $X$, which is defined as 
\[ \mathcal{Z}_{\uu}(X)= \lbrace x\in X \, \vert \, a.x=x.a \, \,  \text{for all}\, \, a\in \uu\rbrace .\] 
\par 
The space of all continuous derivations from $\uu$ into $X$ is denoted by $Z^1(\uu,X)$. A derivation $\delta$ is called \textit{inner derivation}, if there is $x\in X$ such that $\delta(a)=a\cdot x-x\cdot a$ for all $a\in \uu$. Each inner derivation is continuous and $N^1(\uu,X)$ is the space of all inner derivations from $\uu$ into $X$. The \textit{first cohomology group} of $\uu$ with coefficient in $X$ is the quotient space $H^1(\uu,X)=Z^1(\uu,X)/N^1(\uu,X)$. We observe that $X^*$ is a Banach $\uu$-bimodule with the following module operations:
\[\langle a\cdot f, x \rangle=\langle f, x\cdot a \rangle,\quad \langle f\cdot a, x \rangle=\langle f, a\cdot x \rangle\quad\quad (a\in  \uu,  x\in  X, f\in X^*).\]
We call $X^*$ the \textit{dual module} of $X$. Recall that a Banach algebra $\uu$ is said to be \textit{amenable} if for every Banach $\uu$-bimodule $X$, we have $H^1(\uu,X^*)=\lbrace 0\rbrace$. The notion of an amenable Banach algebra was introduced by Johnson in 1972, and it was based on the amenability of locally compact group $G$ \cite{joh1}. In \cite{joh2} (see also \cite[2.9.65]{dal}), it is proved that a Banach algebra $\uu$ is amenable if and only if $\uu$ has a bounded \textit{approximate diagonal}; i.e. a bounded net $\lbrace \T\rbrace_{\lambda\in\Lambda}$ in the projective tensor product $\PT$ that satisfies the following two conditions
\begin{enumerate}
\item[$(1)$] $ a. \T -\T .a  \longrightarrow 0 $;
\item[$(2)$] $ \pi (\T)a \longrightarrow a $
\end{enumerate}
for all $a\in \uu$, where the operations on $\PT$ are defined through 
\[a.(b\otimes c)=ab\otimes c, \quad (b\otimes c).a=b\otimes ca \quad  \text{and} \quad \pi(b\otimes c)=bc \]
for all $a,b,c\in \uu$. The \textit{flip map} on $\PT$ is defined by 
\[ (a\otimes b)^{\circ}=b\otimes a \quad (a,b \in \uu),\]
and an element $\textbf{t}\in \PT$ is called \textit{symmetric} if $\textbf{t}^{\circ}=\textbf{t}$. Johnson in \cite{joh3} is introduced symmetric amenability of Banach algebras. He called a Banach algebra $\uu$ is \textit{symmetrically amenable} if it has a bounded approximate diagonal consisting of symmetric tensors. The properties and examples of symmetrically amenable Banach algebras can be found in \cite{joh3}. 
\par 
Throughout this paper, $\Delta(\uu)$ denotes the character space of $\uu$, that is, all non-zero multiplicative linear functionals on $\uu$. Let  $\phi \in \Delta (\uu)\cup \lbrace 0 \rbrace$. We denote by $\mathcal{M}^{\uu}_{\phi_{r}}$ ($\mathcal{M}^{\uu}_{\phi_{l}}$) the class of all Banach $\uu$-bimodules $X$ with the right (left) module action $x \cdot a = \phi(a)x$ ($a \cdot x = \phi(a)x$) for every $a \in \uu$ and $x \in X$, and $\mathcal{SM}^{\uu}_{\phi} =\mathcal{M}^{\uu}_{\phi_{r}}\cap \mathcal{M}^{\uu}_{\phi_{l}}$. The Banach algebra $\uu$ is called \textit{right $\phi$-amenable} (\textit{left $\phi$-amenable}) if there exists a bounded linear functional $m$ on $\uu^{*}$ satisfying $m(\phi)=1$ and $m(f.a)=\phi(a)m(f)$ ($m(a.f)=\phi(a)m(f)$). This definition was first introduced in \cite{kaniuth} for right $\phi$-amenable Banach algebras. It has been shown that $\uu$ is right $\phi$-amenable if and only if $H^1(\uu,X^*) = \lbrace 0 \rbrace$ for every $X\in \mathcal{M}^{\uu}_{\phi_{l}}$ (\cite[Theorem 1.1]{kaniuth}) if and only if there exists a bounded net $ \lbrace e_{\lambda}\rbrace_{{\lambda\in\Lambda}} $ in $\uu$ such that $ae_{\lambda}-\phi(a)e_{\lambda}\longrightarrow 0$ for all $a\in \uu$ and $ \phi(e_{\lambda})=1 $ for each $ \lambda\in\Lambda $ (\cite[Theorem 1.4]{kaniuth}). A similar result can also be obtained for the left $\phi$-amenability. Indeed, $\uu$ is left $\phi$-amenable if and only if $H^1(\uu,X^*) = \lbrace 0 \rbrace$ for every $X\in \mathcal{M}^{\uu}_{\phi_{r}}$ if and only if there exists a bounded net $ \lbrace e_{\lambda}\rbrace_{{\lambda\in\Lambda}} $ in $\uu$ such that $e_{\lambda}a-\phi(a)e_{\lambda}\longrightarrow 0$ for all $a\in \uu$ and $ \phi(e_{\lambda})=1 $ for each $ \lambda\in\Lambda $. Sangani Monfared has defined the character amenability in \cite{sangani} such that $\uu$ is \textit{character amenable} whenever $H^1(\uu,X^*) = \lbrace 0 \rbrace$ for all $\phi\in \Delta (\uu)\cup \lbrace 0 \rbrace$ and all $X\in\mathcal{M}^{\uu}_{\phi_{r}}\cup\mathcal{M}^{\uu}_{\phi_{l}}$. 
Along these studies, in \cite{sahami}, $\phi $-Johnson amenability is defined. The Banach algebra $\uu$ is called \textit{$\phi $-Johnson amenable}, if there exists an element $m\in (\PT)^{**}$ such that $a.m=m.a$ and $\tilde{\phi}\circ \pi^{**}(m)=1$, for every $a\in \uu$, where $\tilde{\phi}$ is the unique extension of $\phi\in \Delta (\uu)$ on $\uu^{**}$ defined by $\tilde{\phi}(F)=F(\phi)$ for any $F\in \uu^{**}$. The relationship between this concept and other types of amenability will be explained in the next section (Theorem \ref{equivalent phi amenable}).
\par 
Suppose that $\ell_{\infty}(\Lambda)$ is the Banach space of all bounded nets $\lbrace z_{\lambda} \rbrace_{\lambda\in\Lambda}$ of complex
numbers, and $c(\Lambda)$ is the closed subspace of $\ell_{\infty}(\Lambda)$ consisting of all convergent nets. The mapping 
\[ \lbrace z_{\lambda} \rbrace_{\lambda\in\Lambda}\mapsto \lim _{\lambda\in \Lambda} z_{\lambda}\]
is a continuous linear functional with norm one on $c(\Lambda)$. By the Hahn–Banach theorem there exists a norm-preserving extension of $\lim _{\lambda\in \Lambda}$ to $\ell_{\infty}(\Lambda)$, which is called a \textit{generalized limit} on $\Lambda$, and we will denote by $Lim_{\lambda\in \Lambda}$.
\section{Some useful technical results}
In this section, we present some results about $\phi $-Johnson amenable Banach algebras, which are used to prove our main results, and these results can be interesting in themselves. First, in the following theorem, we provide some conditions equivalent to $\phi $-Johnson amenability.
 \begin{thm}\label{equivalent phi amenable}
  Let $\uu $ be a Banach algebra and  $ \phi \in \Delta (\uu ) $. The following are equivalent.
 \begin{itemize}
  \item[(i)]
 $ \uu $ is $\phi  $-Johnson amenable;
  \item[(ii)]
  there is a bounded net $\lbrace\T\rbrace_{\lambda\in\Lambda} $  in $ \PT $ such that
 \[  a.\T-\T.a \longrightarrow 0 \quad \text{and} \quad \phi \circ\pi(\T)\longrightarrow 1 \]
 for each $ a\in\uu $;
 \item[(iii)]
 $ \uu $ is right and left $ \phi $-amenable;
 \item[(iv)]
 there is a bounded net $ \lbrace e_{\lambda}\rbrace_{{\lambda\in\Lambda}} $ in  $\uu $ such that
 \[ ae_{\lambda}-\phi(a)e_{\lambda}\longrightarrow 0 \quad \text{and} \quad e_{\lambda} a-\phi(a)e_{\lambda}\longrightarrow 0 \]
 for each $ a\in\uu $ and $ \phi(e_{\lambda})=1 $ for each $ \lambda\in\Lambda $;
 \item[(v)]
 there is a symmetric bounded net $\lbrace\T\rbrace_{\lambda\in\Lambda} $  in $ \PT $ such that
 \[  a.\T-\T.a \longrightarrow 0 \quad \text{and} \quad \phi \circ\pi(\T)\longrightarrow 1 \]
 for each $ a\in\uu $;
 \item[(vi)]
 there is a continuous linear functional $ m $ on $ \uu^{*} $ such that 
 \[  m(\phi)=1  \quad \text{and} \quad  m(f.a)=m(a.f)=\phi(a) m(f) \]
 for each  $ a\in\uu $ and $ f\in\uu^{*}$.
 \end{itemize}
 \end{thm}
 \begin{proof}
 (i)$\Leftrightarrow$(ii): It is obtained from \cite[Lemma 2.1]{sahami}.\\
 (ii)$\Leftrightarrow$(iii): It is obtained from \cite[Proposition 2.2]{sahami}.\\
 (iii)$\Rightarrow$(iv): From \cite[Theorem $1.4$]{kaniuth} there are bounded nets $ \lbrace e_{\alpha}\rbrace_{\alpha\in I} $ and
 $\lbrace e_{\beta}\rbrace_{\beta\in J} $
 such that $ \phi (e_{\alpha})=\phi (e_{\beta})=1 $ for each $ \alpha \in I , \beta\in J $ and
 \[\|ae_{\alpha}-\phi(a)e_{\alpha}\| \longrightarrow 0  \quad \text{and} \quad \|e_{\beta}a-\phi(a)e_{\beta}\|\longrightarrow 0\]
for each $ a\in\uu $. Suppose that $ \Lambda=I\times J $ and for each $ \lambda=(\alpha,\beta)\in\Lambda $ define $e_{\lambda}=e_{\alpha}e_{\beta} $. The net $\lbrace e_{\lambda}\rbrace_{\lambda\in\Lambda} $ is bounded and
 $ \phi(e_{\lambda})=\phi(e_{\alpha}) \phi(e_{\beta})=1$ for all $ \lambda\in\Lambda$. Also, we have 
 \[ \|ae_{\lambda}-\phi(a)e_{\lambda}\|=\|(ae_{\alpha}-\phi(a)e_{\alpha})e_{\beta}\| \leq\|ae_{\alpha}-\phi(a)e_{\alpha}\|\|e_{\beta}\|\longrightarrow 0 \]
 and 
\[  \|e_{\lambda}a-\phi(a)e_{\lambda}\|=\|e_{\alpha}(e_{\beta}-\phi(a)e_{\beta})\|\leq\|e_{\alpha}\|\|e_{\beta}a-\phi(a)e_{\beta}\|\longrightarrow 0\]
for each $ a\in\uu $.\\
(iv)$\Rightarrow$(v): For each $ \lambda\in\Lambda $ define $\T=e_{\lambda} \otimes e_{\lambda} \in \PT $. In this case $ \lbrace \T\rbrace_{\lambda}\in\Lambda $ is a bounded net in $\PT $ such that
\[ \phi \circ\pi(\T)=\phi(e_{\lambda})\phi(e_{\lambda})=1 \quad \quad (\lambda\in\Lambda) \]
and for each $ \lambda\in\Lambda $ we have $ \T^{\circ}  =\T  $. Also 
\begin{equation*}
\begin{split}
\|a. \T -\T .a\|&=\|ae_{\lambda}\otimes e_{\lambda}-e_{\lambda}\otimes e_{\lambda}a\| \\ &
\leq\|ae_{\lambda}\otimes e_{\lambda}-\phi(a)e_{\lambda}\otimes e_{\lambda}\|+\|\phi(a)e_{\lambda}\otimes e_{\lambda}-e_{\lambda}\otimes e_{\lambda}a\| \\&
\leq\|ae_{\lambda}-\phi(a)e_{\lambda}\|\|e_{\lambda}\|+\|e_{\lambda}a-\phi(a)e_{\lambda}\| \|e_{\lambda}\| \longrightarrow 0
\end{split}
\end{equation*}  
for each $ a\in\uu $. \\
(v)$\Rightarrow$(ii): is clear.\\
(iv)$\Rightarrow$(vi): Choose a subnet $\lbrace e_{\lambda_{\gamma}} \rbrace_{\gamma\in\Omega}$ of $\lbrace e _{\lambda}\rbrace_{\lambda\in\Lambda} $ that is converges to $ m\in\uu^{**} $ in the $ w^{*}$-topology of $\uu^{**} $. So
\[ m(\phi)=\lim_{\gamma} e_{\lambda_{\gamma}}(\phi)=\lim_{\gamma} \phi(e_{\lambda_{\gamma}})=1. \]
We have
 \[ a.m-\phi(a)m=w^{*}-\lim _{\gamma}(ae_{\lambda_{\gamma}}-\phi(a)e_{\lambda_{\gamma}})=0 \] 
 and 
 \[m.a-\phi(a)m=w^{*}-\lim _{\gamma}(e_{\lambda_{\gamma}}a-\phi(a)e_{\lambda_{\gamma}})=0 \]
 for each $ a\in\uu $. Hence 
 \[  m(f.a)=\phi(a)m(f) \quad \text{and} \quad m(a.f)=\phi(a)m(f) \]
 for each $ a\in\uu $ and $ f\in\uu^{*} $.\\
 (vi)$\Rightarrow$(iii): It is clear according to the definition of right and left $ \phi$-amenability.
 \end{proof}
 In view of this theorem it is clear that if $ \uu $ is character ameneble, then $ \uu $ is $ \phi $-Johnson amenable for every $ \phi\in\Delta(\uu) $.
 \par 
 If we want to define symmetric $ \phi $-Johnson amenability based on the definition of symmetric amenability, according to the previous theorem, we see that this definition is equivalent to $ \phi $-Johnson amenability.
 \par 
 Let $ X $ be a bimodule over the algebra $ \uu $. The derivation $ \delta:\uu\rightarrow X $ is called a \textit{central derivation} whenever $ \delta(\uu)\subseteq \mathcal{Z}_{\uu}(X) $. It is proved in \cite[Proposition 1]{feizi} that if $\uu$ is right or left $\phi$-amenable and $X\in \mathcal{SM}^{\uu}_{\phi}$, then $H^{1}(\uu , X)=\lbrace 0 \rbrace$. In the next lemma, we determine the central derivations on right or left $ \phi $-amenable Banach algebras.
 \begin{lem}\label{central derivation}
 Let $ \uu $ be a Banach algebra, $ \phi\in\Delta(\uu) $, and $ \uu $ be right or left $ \phi $-amenable. If $ X\in\mathcal{M}^{\uu}_{\phi_{r}}\cup\mathcal{M}^{\uu}_{\phi_{l}} $ and $ \delta:\uu\rightarrow X $ is a continuous central derivation, then $ \delta=0 $.
 \end{lem}
 \begin{proof} 
 Suppose that $ \uu $ is left $ \phi $-amenable. For each $a\in\uu$ and $ x\in\mathcal{Z}_{\uu}(X) $ we have $a.x=x.a=\phi(a)x $. So $ \mathcal{Z}_{\uu}(X) $ is a Banach $ \uu $-bimodule such that $ \mathcal{Z}_{\uu}(X)\in \mathcal{SM}^{\uu}_{\phi} $. By \cite[Proposition 1]{feizi} the derivation $ \delta:\uu\rightarrow \mathcal{Z}_{\uu}(X) $ is inner, and hence $ \delta=0 $.\par 
If $ \uu$ is right $ \phi $-amenable the proof is similar.
 \end{proof}
The proof of the following lemma is derived from \cite[Lemma 3]{alaminos}. 
\begin{lem}\label{maps into submodule}
Let $ \uu $ be a right or left $ \phi $-amenable Banach algebra where $ \phi\in\Delta(\uu) $. Suppose that  $ Y\in\mathcal{M}^{\uu}_{\phi_{r}}\cup\mathcal{M}^{\uu}_{\phi_{l}} $ and $ X $ is a closed $ \uu $-subbimodule of $ Y $. If
 $ \delta:\uu\rightarrow Y $ is a continuous derivation and 
 $ \tau:\uu\rightarrow \mathcal{Z}_{\uu}(Y)$ is a linear map such that $ (\delta+\tau)(\uu)\subseteq X $, then $ \delta(\uu)\subseteq X $ and $ \tau(\uu)\subseteq\mathcal{Z}_{\uu}(X) $.
 \end{lem}
 \begin{proof}
 Let $ \pi:Y\rightarrow \dfrac{Y}{X} $ be the quotient map where $ W=\dfrac{Y}{X} $ is the quotient Banach $ \uu $-bimodule. By our assumption
 $ X,W\in \mathcal{M}^{\uu}_{\phi_{r}}\cup\mathcal{M}^{\uu}_{\phi_{l}}$ and
\[0=\pi\circ(\delta+\tau)=\pi\circ\delta+\pi\circ\tau .\]
 Therefore 
\[\pi\circ\delta=-\pi\circ\tau . \]
Since $ \pi $ maps $ \mathcal{Z}_{\uu}(Y) $ into $ \mathcal{Z}_{\uu}(W) $ and $ \tau(\uu)\subseteq\mathcal{Z}_{\uu}(Y) $, it follows that 
$ \pi\circ\delta(\uu)=-\pi\circ\tau(\uu)\subseteq\mathcal{Z}_{\uu}(W) $. So by the fact that $ \pi $ is a continuous module homomorphism, $ \pi\circ\delta$ is a continuous central derivation from $ \uu $ into $ W $. According to Lemma \ref{central derivation} $ \pi\circ\delta=0 $, and hence $\delta(\uu)\subseteq X $. Now from assumption and the obtained result, we have $ \tau(\uu)\subseteq X\cap\mathcal{Z}_{\uu}(Y)=\mathcal{Z}_{\uu}(X)$.
\end{proof}
According to Theorem \ref{equivalent phi amenable}, especially in the case where the Banach algebra $ \uu $ is $ \phi $-Johnson amenable, Lemmas \ref{central derivation} and \ref{maps into submodule} are established.
\section{Jordan derivations}
In this section, we study the continuous Jordan derivations of $ \phi $-Johnson amenable Banach algebras into special Banach modules. The following theorem is the main result of this section.
\begin{thm}\label{jordan}
Let $ \uu $ be a $ \phi $-Johnson amenable Banach algebra and $ X\in\mathcal{M}^{\uu}_{\phi_{r}}\cup\mathcal{M}^{\uu}_{\phi_{l}} $. Then every continuous Jordan derivation from $ \uu $ to $ X $ is a derivation.
\end{thm}
\begin{proof}
 Suppose that $ X\in\mathcal{M}^{\uu}_{\phi_{r}} $ and $ \delta:\uu\rightarrow X $ is a continuous Jordan derivation. So $ X^{*}\in\mathcal{M}^{\uu}_{\phi_{l}} ,\quad X^{**}\in\mathcal{M}^{\uu}_{\phi_{r}}$. Viewing $ X $  as a closed $\uu$-subbimodule of $X^{**}$, and hence $ \delta $ is a continuous Jordan derivation from $ \uu $ to $ X^{**} $.\par
 According to Theorem \ref{equivalent phi amenable} there is a bounded net $ \lbrace e_{\lambda}\rbrace_{\lambda\in\Lambda} $ in $\uu$
 such that 
\begin{equation}\label{q1}
\| ae_{\lambda}-\phi(a)e_{\lambda}\|\longrightarrow 0 \quad \text{and} \quad \|e_{\lambda}a-\phi(a)e_{\lambda}\|\longrightarrow 0 
\end{equation}
for each $ a\in\uu $ and $ \phi(e_{\lambda})=1 $ for each $ \lambda\in\Lambda $. Define $ \Omega\in X^{**} $ as follows:
\[ \langle\Omega,f\rangle=Lim_{\lambda}\langle\delta(e_{\lambda}),f\rangle, \]
 where $f\in X^{*}  $ and $ Lim_{\lambda} $ is a generalized limit on $ \Lambda $. From proof of Theorem \ref{equivalent phi amenable}, $\lbrace e_{\lambda}\otimes e_{\lambda}\rbrace _{\lambda\in\Lambda} $ is a symmetric bounded net in $\PT$ such that 
 \begin{equation}\label{q2}
  ae_{\lambda}\otimes e_{\lambda}-e_{\lambda}\otimes e_{\lambda}a\longrightarrow 0 \quad \text{and} \quad \phi(\pi(e_{\lambda}\otimes e_{\lambda}))=1
 \end{equation}
 for each $a\in \uu$. Since the flip map on $ \PT $ is continuous, it follows that 
 \begin{equation}\label{q3}
   e_{\lambda}\otimes ae_{\lambda}-e_{\lambda}a\otimes e_{\lambda}\longrightarrow 0 
 \end{equation}
for each $a\in \uu$. Define the continuous bilinear map from $ \uu \times \uu$ to $ X $ by $ (a,b)\mapsto \phi(a)\delta(b) $, and therefore there is a continuous linear map $ \Psi : \PT\rightarrow X $ such that $ \Psi (a\otimes b)=\phi(a)\delta(b) $ for each $ a,b\in\uu$. According to \eqref{q2} and \eqref{q3} for each $ a\in\uu $ and $  f\in X^{*}$ we have
 \[ Lim_{\lambda}\langle\Psi(ae_{\lambda}\otimes e_{\lambda}),f\rangle=Lim_{\lambda}\langle\Psi(e_{\lambda}\otimes e_{\lambda}a),f\rangle \]
 and
 \[  Lim_{\lambda}\langle\Psi(e_{\lambda}\otimes ae_{\lambda}),f\rangle=Lim_{\lambda}\langle\Psi(e_{\lambda}a\otimes e_{\lambda}),f\rangle . \]
 So by the fact that $ \phi(e_{\lambda})=1 $, we see that 
 \begin{equation}\label{q4}
 Lim_{\lambda}\langle\phi (a)\delta (e_{\lambda}),f\rangle=Lim_{\lambda}\langle\delta(e_{\lambda}a).f\rangle
 \end{equation}
 and
 \begin{equation}\label{q5}
 Lim_{\lambda}\langle\delta(ae_{\lambda}),f\rangle=Lim_{\lambda}\langle\phi(a)\delta(e_{\lambda}),f\rangle 
 \end{equation}
 for each $ a\in\uu $ and $  f\in X^{*}$. Now according to \eqref{q4} and \eqref{q5} for each $ a\in\uu$ and $f\in X^{*} $ we have
 \begin{equation*}
 \begin{split}
 \langle\phi(a)\Omega,f\rangle &=\langle\Omega.a,f\rangle \\&
 =\langle\Omega,a.f\rangle \\&
 =Lim_{\lambda}\langle\delta(e_{\lambda}).a,f\rangle\\&
 =Lim_{\lambda}\langle\phi(a)\delta(e_{\lambda}),f\rangle \\&
 =Lim_{\lambda}\langle\delta(ae_{\lambda}),f\rangle \\&
 =Lim_{\lambda}\langle a.\delta(e_{\lambda})+\delta(a).e_{\lambda}+\delta(e_{\lambda}).a+e_{\lambda}.\delta(a)-\delta(e_{\lambda}a),f\rangle \\&
 =Lim_{\lambda}\langle a.\delta(e_{\lambda})+\phi(e_{\lambda})\delta(a)+\phi(a)\delta(e_{\lambda})+e_{\lambda}.\delta(a)-\delta(e_{\lambda}a),f\rangle \\&
 =Lim_{\lambda}\langle a.\Omega+\delta(a)+\Delta(a),f\rangle,
 \end{split}
 \end{equation*}
 where $ \Delta:\uu\rightarrow X^{**} $ is a continuous linear map defined by
\[\langle\Delta(a),f\rangle=Lim_{\lambda}\langle e_{\lambda}.\delta(a),f\rangle\quad\quad (a\in\uu, f\in X^{*} ).\]
So
\begin{equation}\label{q6}
\Delta(a)=\Omega.a-a.\Omega-\delta(a)
\end{equation}
for each $a\in\uu$, and hence $ \Delta $ is a Jordan derivation. By \eqref{q1} we get
\begin{equation*}
 \begin{split}
 \langle a.\Delta(b),f\rangle &=Lim_{\lambda}\langle ae_{\lambda}.\delta(b),f\rangle \\&
 =Lim_{\lambda}\langle \phi(a)e_{\lambda}.\delta(b),f\rangle \\&
 =\langle\phi(a)\Delta(b),f\rangle \\&
 =\langle\Delta(b).a,f\rangle
 \end{split}
 \end{equation*}
for each $ a,b\in\uu$ and $f\in X^{*} $. Thus
\begin{equation}\label{q7}
a.\Delta(b)=\Delta(b).a=\phi(a)\Delta(b) 
\end{equation} 
 for each $ a,b\in \uu $. Now we do the same process for $ \Delta $ as we did earlier for $ \delta $ and therefore
 \begin{equation}\label{q8}
   \Delta(a)=\Omega_{1}.a-a.\Omega_{1}-\Delta_{1}(a) \quad \quad (a\in\uu),
 \end{equation}
where $ \Omega_{1}\in X $ and $ \Delta_{1} $ is a continuous linear map from $ \uu $ to $ X^{**} $ defined by
\[  \langle\Delta_{1}(a),f\rangle=Lim_{\lambda}\langle e_{\lambda}.\Delta(a),f\rangle \quad\quad (a\in\uu, f\in X^{*} ).\]
From \eqref{q7} it follows that 
\begin{equation*}
 \begin{split}
 \langle\Delta_{1}(a),f\rangle &=Lim_{\lambda}\langle\phi(e_{\lambda})\Delta(a),f\rangle \\&
 = Lim_{\lambda}\langle\Delta(a),f\rangle \\&
 =\langle\Delta(a),f\rangle 
  \end{split}
 \end{equation*}
 for each $a\in\uu$ and $ f\in X^{*}$. So $\Delta_{1}=\Delta$, and from \eqref{q8} we arrive at 
  \[ \Delta(a)=(\dfrac{1}{2}\Omega_{1}).a-a.(\dfrac{1}{2}\Omega_{1})\quad \quad (a\in\uu). \]
 According to this identity and \eqref{q6} we have
 \[\delta(a)=(\Omega-\dfrac{1}{2}\Omega_{1}).a-a.(\Omega-\dfrac{1}{2}\Omega_{1})\quad \quad(a\in\uu),\]
and hence  $ \delta $ is a derivation. \par 
 If $ X\in\mathcal{M}^{\uu}_{\phi_{l}} $, it is proved similarly.
\end{proof}
From Theorem \ref{equivalent phi amenable} and the definition of character amenability, it is clear that if $ \uu $ is a character ameneble Banach algebra, then $ \uu $ is $ \phi $-Johnson amenable for every $ \phi\in\Delta(\uu) $. Considering this point and Theorem \ref{jordan}, we have the following result.
\begin{cor}\label{Jordan character amenable}
Suppose that $ \uu$ is a character amenable Banach algebra. Then for each $ \phi\in\Delta(\uu) $ and $ X\in\mathcal{M}^{\uu}_{\phi_{r}}\cup\mathcal{M}^{\uu}_{\phi_{l}} $, every continuous Jordan derivation from $ \uu $ to $ X $ is a derivation.
\end{cor}
Examples of character amenable Banach algebras are presented in \cite{sangani}. We also know that every amenable Banach algebra is character amenable. Therefore, we have the following result.
\begin{cor}\label{Jordan amenable}
Suppose that the Banach algebra $ \uu $ is amenable. Then for each $ \phi\in\Delta(\uu) $ and $ X\in\mathcal{M}^{\uu}_{\phi_{r}}\cup\mathcal{M}^{\uu}_{\phi_{l}} $, every continuous Jordan derivation from $ \uu $ to $ X $ is a derivation.
\end{cor}
Suppose that $ \phi\in\Delta(\uu) $ and $ \mathbb{C}\in\mathcal{SM}^{\uu}_{\phi} $. In this case, a derivation $\delta:\uu\rightarrow\mathbb{C}  $ is called \textit{point derivation} at $ \phi $ and a Jordan derivation $\delta:\uu\rightarrow\mathbb{C}  $ is called \textit{point Jordan derivation} at $ \phi $. Indeed, if $ \delta $ is a point derivation at $ \phi $, then 
\[ \delta(ab)=\phi(a)\delta(b)+\phi(b)\delta(a)\quad\quad (a,b\in\uu), \]
and if $\delta$ is a point Jordan derivation at $ \phi $ then
\[ \delta(a^{2})=2\phi(a)\delta(a)\quad(a\in\uu) .\]
Considering that $ \mathcal{SM}^{\uu}_{\phi}=\mathcal{M}^{\uu}_{\phi_{r}}\cap\mathcal{M}^{\uu}_{\phi_{l}} $, and in view of Lemma \ref{central derivation} and the previous results, we have the following result.
\begin{cor}\label{point Jordan der}
Let $\uu$ be a Banach algebra.
\begin{itemize}
\item[(i)]
If $\uu $ is $ \phi$-Johnson amenable, then every continuous point Jordan derivation at $ \phi $ on $ \uu $ is equal to zero.
\item[(ii)]
If $ \uu $ is character amenable, then for each $\phi\in\Delta(\uu)$ every continuous point Jordan derivation at $ \phi $ on $ \uu $ is equal to zero.
\item[(iii)]
If $ \uu $ is amenable, then for each $\phi\in\Delta(\uu)$ every continuous point Jordan derivation at $ \phi $ on $ \uu $ is equal to zero.
\end{itemize}
\end{cor}
This result actually gives us a necessary condition for $ \phi$-Johnson amenability, character amenability and amenability based on point Jordan derivations.
\par 
The following example shows that the $ \phi $-Johnson amenability condition cannot be dropped from Theorem \ref{jordan}.
\begin{exam}\label{example jordan der}
We consider the Banach algebra $ \uu $ and the character $ \phi \in\Delta(\uu)$ as follows
\[\uu=\left\lbrace\begin{bmatrix}
0&a\\
0&b 
\end{bmatrix}
\mid a,b\in\mathbb{C}\right\rbrace \quad  \text{and} \quad  \phi(\begin{bmatrix}
0&a\\
0&b 
\end{bmatrix})=b . 
\]
The element
$ E=\begin{bmatrix}
0&0\\
0&1
\end{bmatrix} \in\uu $
is a right identity of $ \uu $ and for every 
$A= \begin{bmatrix}
0&a\\
0&b
\end{bmatrix} \in\uu $ we have
\[ EA=\phi(A)E\quad  \text{and} \quad \phi(E)=1 \]
Therefore $ \uu $ is left $ \phi$-amenable. According to \cite[Example 2.5]{sahami}, $ \uu $ is not $ \phi$-Johnson amenable. Now we turn $ X:=\mathbb{C} $ into a Banach $ \uu $-bimodule by the module actions
\[ A.\lambda=\phi(A)\lambda \quad  \text{and} \quad  \lambda.A=0\quad\quad (A\in\uu,\lambda\in\mathbb{C}).\]
In this case $ X\in\mathcal{M}^{\uu}_{\phi_{l}} $. Define the linear map $ \delta:\uu\rightarrow X $ by
\[
 \delta(\begin{bmatrix}
0&a\\
0&b
\end{bmatrix})=a\quad\quad (\begin{bmatrix}
0&a\\
0&b
\end{bmatrix}\in\uu) .\]
A routine calculation shows that $ \delta $ is a Jordan derivation, which is not a derivation.
\par 
In the continuation of the example, we consider the Banach algebra $ \mathcal{G}$ and $ \phi \in\Delta(\mathcal{G})$ as follows
\[\mathcal{G}=\left\lbrace\begin{bmatrix}
a&b\\
0&0
\end{bmatrix}\mid a,b\in\mathbb{C} \right\rbrace \quad  \text{and} \quad \phi( \begin{bmatrix}
a&b\\
0&0
\end{bmatrix})=a . \]
Then $ \mathcal{G} $ is a right $ \phi $-amenable Banach algebra.
By the same reasoning as \cite[Example 2.5 ]{sahami}, $ \mathcal{G} $ is not $ \phi $-Johnson amenable.
$ Y:=\mathbb{C} $ is a Banach $ \mathcal{G} $-bimodule by the module actions 
\[A.\lambda=0\quad  \text{and} \quad \lambda.A=\phi(A)\lambda \quad\quad(A\in\mathcal{G},\lambda\in\mathbb{C}), \]
and hence $ Y\in\mathcal{M}^{\mathcal{G}}_{\phi_{r}} $.
The linear map
$ \delta:\mathcal{G}\rightarrow X $
defined by
\[
 \delta(\begin{bmatrix}
a&b\\
0&0
\end{bmatrix} )=b\quad\quad  (\begin{bmatrix}
a&b\\
0&0
\end{bmatrix}\in\mathcal{G})
\]
is a Jordan derivation, which is not derivation.
\end{exam}
This example not only shows that $ \phi $-Johnson amenability condition cannot be dropped, but also shows that the condition cannot be replaced by the weaker condition left or right $ \phi$-amenabiliy.
\section{Lie derivations}
In this section, we study continuous Lie derivations on $ \phi$-Johnson amenable Banach algebras. The following theorem is the main result of this section.
\begin{thm}\label{Lie} 
Let $ \uu $ be a $ \phi$-Johnson amenable Banach algebra where $\phi \in\Delta(\uu)$, and let $ X\in \mathcal{M}^{\uu}_{\phi_{r}}\cup\mathcal{M}^{\uu}_{\phi_{l}} $.
Suppose that $ \delta:\uu\rightarrow X $ is a continuous Lie derivation. Then there exist a continuous derivation $ d:\uu\rightarrow X $ and a continuous center-valued trace $\varphi:\uu\rightarrow\mathcal{Z}_{\uu}(X)  $ such that $ \delta=d+\varphi $.
\end{thm}
\begin{proof}
Assume that $ X\in\mathcal{M}^{\uu}_{\phi_{r}} $ and $ \delta:\uu\rightarrow X $ is a continuous Lie derivation. We consider $ \delta $ as a continuous Lie derivation from $ \uu $ to $ X^{**} $. As in the proof of Theorem \ref{jordan}, relations \eqref{q1} to \eqref{q5} are also valid here, and we also consider $ \Omega\in X^{**} $ as in the proof of Theorem \ref{jordan}. According to \eqref{q4} and \eqref{q5}, for each $a\in\uu , f\in X^{*} $ we have
\begin{equation*}
\begin{split} 
\langle\phi(a)\Omega,f\rangle=\langle\Omega.a,f\rangle &=\langle\Omega,a.f\rangle \\&
=\langle\Omega,a.f\rangle \\&
=Lim_{\lambda}\langle\delta(e_{\lambda}).a,f\rangle \\&
=Lim_{\lambda}\langle \delta(ae_{\lambda}),f\rangle \\&
=Lim_{\lambda}\langle a.\delta(e_{\lambda})+\delta(a)e_{\lambda}-e_{\lambda}.\delta(a)-\delta(e_{\lambda}).a+\delta(e_{\lambda}a),f\rangle \\&
=Lim_{\lambda}\langle a.\delta(e_{\lambda})+\delta(a)-e_{\lambda}.\delta(a)-\phi(a)\delta(e_{\lambda})+\delta(e_{\lambda}a),f\rangle\\&
=Lim_{\lambda}\langle a.\Omega+\delta(a)-\varphi(a),f\rangle,
\end{split}
\end{equation*}
where $ \varphi:\uu\rightarrow X^{**} $ is a linear map defined by
\[ \langle\varphi (a),f\rangle =Lim_{\lambda}\langle e_{\lambda}.\delta (a),f\rangle\quad\quad (a\in\uu , f\in X^{*}). \]
 So
\[ \delta(a)=\Omega.a-a.\Omega+\varphi(a)\quad\quad (a\in\uu). \]
The linear map $d:\uu\rightarrow X^{**}  $ defined by $d(a)=\Omega.a-a.\Omega $ is a continuous derivation, and therefore $\delta=d+\varphi $. Also, $  \varphi(a)\in\mathcal{Z}_{\uu}(X^{**})$ for every $ a\in\uu $ (the proof is the same as the proof of centrality of $ \Delta $ in the proof of Theorem \ref{jordan}). Now $ \varphi=\delta-d $ is a continuous Lie derivation, and from the fact that $ \varphi(\uu)\subseteq\mathcal{Z}_{\uu}(X^{**}) $, it follows that $  \varphi([a,b])=0$ for every $ a,b\in\uu $. The conditions of Lemma \ref{maps into submodule} hold for $ d$ and $\varphi $, hence $ d$ maps $\uu$ to $X $ and $ \varphi$ maps $\uu$ to $\mathcal{Z}_{\uu}(X) $.
\end{proof}
We have the following results from the above theorem.
\begin{cor}\label{Lie character amenable}
Let $ \uu $ be a character amenable Banach algebra. Then for each $ \phi\in\Delta(\uu) $ and $X\in\mathcal{M}^{\uu}_{\phi_{r}}\cup\mathcal{M}^{\uu}_{\phi_{l}} $, for every continuous Lie derivation $\delta:\uu\rightarrow X$, there exist a continuous derivation $ d:\uu\rightarrow X $ and a continuous center-valued trace $ \varphi:\uu\rightarrow\mathcal{Z}_{\uu}(X) $ such that $ \delta=d+\varphi $ for each $ a,b\in\uu $.
\end{cor}
\begin{cor}\label{Lie amenable}
Let $ \uu $ be a amenable Banach algebra. Then for each $ \phi\in\Delta(\uu) $ and $X\in\mathcal{M}^{\uu}_{\phi_{r}}\cup\mathcal{M}^{\uu}_{\phi_{l}} $, for every continuous Lie derivation $\delta:\uu\rightarrow X$, there exist a continuous derivation $ d:\uu\rightarrow X $ and a continuous center-valued trace $ \varphi:\uu\rightarrow\mathcal{Z}_{\uu}(X) $ such that $ \delta=d+\varphi $ for each $ a,b\in\uu $.
\end{cor}
Let $\mathbb{C}\in\mathcal{SM}^{\uu}_{\phi} $. In this case, like the point Jordan derivation, we can also define the point Lie derivation $ \delta:\uu\rightarrow\mathbb{C} $ at $ \phi $. So the linear map $ \delta:\uu\rightarrow\mathbb{C} $ is a point Lie derivation at $ \phi $ if and only if $ \delta([a,b])=0 $ for every $ a,b\in\uu $. Therefore, if we put $ V=\overline{span\lbrace[a,b]\mid a,b\in\uu\rbrace} $, then the set of all continuous point derivations at $ \phi $ on $ \uu $ is equal to $ V^{\perp} $. Unlike point Jordan derivations, the existence of a non-zero continuous point Lie derivation at $ \phi $ on $\uu $, does not result $ \phi$-Johnson non-amenability of $ \uu $. For example if $ \uu\neq 0$ is right (or left) $ \phi $-amenable commutative Banach algebra, according to Theorem \ref{equivalent phi amenable}, it is clear that $\uu $ is $ \phi $-Johnson amenable, and also the closed linear subspace $ V $ of $ \uu $ as defined above is $ V=0 $. So $ V^{\perp}=\uu^{*} $ and therefore $ \uu $ has non-zero continuous point Lie derivations at $ \phi $. An example of $ \phi $-Johnson amenable commutative Banach algebra is as follows. Suppose that $\mathbb{D} $ is the open unit disk, and $ A^{+}(\overline{\mathbb{D}}) $ is the set of all functions $f=\sum^{\infty}_{n=0}c_{n}Z^{n}  $ in the disc algebra $A(\overline{\mathbb{D}}) $ which have an absolutely convergent Taylor expansion on $\overline{\mathbb{D}}$. With respect to pointwise operations and the norm $ \Vert\sum^{\infty}_{n=0}c_{n}Z^{n}\Vert=\sum^{\infty}_{n=0}\vert c_{n}\vert $, $A^{+}(\overline{\mathbb{D}}) $ is a commutative Banach algebra. The map $z\rightarrow \varphi_{z}  $ where $ \varphi_{z} $ is a point evaluation at $ z $, is a character on $ A^{+}(\mathbb{D}) $. If $ \vert z\vert=1 $, then $A^{+}(\overline{\mathbb{D}}) $ is (right) $ \varphi_{z} $-amenable, and hence $A^{+}(\overline{\mathbb{D}}) $ is $ \varphi_{z} $-Johnson amenable (see \cite[Example 2.5]{kaniuth}).
\par 
The following example shows that the condition $ \phi $-Johnson amenability from Theorem \ref{Lie} cannot be removed or replaced by a weaker condition right or left $ \phi $-amenability.
\begin{exam}\label{example Lie der}
Consider the Banach algebras $ \uu $ and $ \mathcal{G} $ defined in Example \ref{example jordan der} and the character $ \phi $ defined on them.
\begin{itemize}
\item[$ (i)$]
We turn $ X:=\mathbb{C} $ into a Banach $ \uu $-bimodule by the module actions
\[ A.\lambda=0 \quad  \text{and} \quad \lambda.A=\phi(A)\lambda\quad \quad(A\in\uu,\lambda\in\mathbb{C}).\]
In this case $ X\in\mathcal{M}^{\uu}_{\phi_{r}} $ and $ \mathcal{Z}_{\uu}(X)=0 $. The linear map $ \delta:\uu\rightarrow X $ defined by 
$ \delta(\begin{bmatrix}
0&a\\
0&b
\end{bmatrix})=a$
is a Lie derivation which is not derivation. So, $ \delta $ is not as mentioned in Theorem \ref{Lie}.
\item[$(ii)$]
We turn $Y:=\mathbb{C} $ into a Banach $\mathcal{G}  $-bimodule by the module actions
\[\lambda.A=0\quad  \text{and} \quad A.\lambda=\phi(A).\lambda\quad\quad (A\in\uu,\lambda\in\mathbb{C}).\]
In this case 
$ Y\in\mathcal{M}^{\uu}_{\phi_{l}} $ and $ \mathcal{Z}_{\uu}(Y)=0 $. The linear map $ \delta:\mathcal{G}\rightarrow Y $ defined by $\delta(\begin{bmatrix}
a&b\\
0&0
\end{bmatrix})=b$
is a Lie derivation which is not derivation. Hence, $ \delta $ is not as mentioned in Theorem \ref{Lie}.
\end{itemize}
\end{exam}

\subsection*{Declarations}
\begin{itemize}
\item[•] \textbf{Author's contribution:} All authors contributed to the study conception and design and approved the final manuscript. 
\item[•] \textbf{Funding:} The author declares that no funds, grants, or other support were received during the preparation of this manuscript.
\item[•] \textbf{Conflict of interest:} On behalf of all authors, the corresponding author states that there is no conflict of interest. 
\item[•] \textbf{Data availability:} Data sharing not applicable to this article as no datasets were generated or analysed during the current study. 
\end{itemize}

\subsection*{Acknowledgment}
The authors like to express their sincere thanks to the referee(s) for this paper.



\bibliographystyle{amsplain}
\bibliography{xbib}

\begin{thebibliography}{20}

\bibitem{alaminos0}
J. Alaminos, M. Bre\v{s}ar, A.R. Villena, \textit{The strong degree and the structure of Lie and Jordan derivations from von Neumann algebras}, Math. Proc. Camb. Phil. Soc. 137 (2004), 441--463.

\bibitem{alaminos}
J. Alaminos, M. Mathieu and A.R. Villena, \textit{Symmetric amenability and Lie derivations}, Math. Proc. Camb. Phil. Soc. 137 (2004), 433--439.

\bibitem{beh}	
R. Behfar and H. Ghahramani, \textit{Lie maps on triangular algebras without assuming unity}, Mediterranean J. Math. 18 (2021), 1--28.

\bibitem{cus}
J.M. Cusack, \textit{Jordan derivations on rings}, Proc. Amer. Math. Soc. 53 (1975), 321--324.

\bibitem{dal}
H. G. Dales. \textit{Banach algebras and automatic continuity}, London Mathematical Society Monographs New Series 24. (Clarendon Press, 2000).


\bibitem{feizi}
E. Feizi, H. Ghahramani and V. Khodakarami, \textit{The first and second Hochschild cohomology groups of Banach algebras with coefficients in special symmetric bimodules }, Complex Anal. Oper. Theory, 14, 68 (2020). https://doi.org/10.1007/s11785-020-01027-w.


\bibitem{gh0}
H Ghahramani, M.N. Ghosseiri and L. Heidari Zadeh, \textit{On the Lie derivations and generalized Lie derivations of quaternion rings}, Commun. Algebra, 47 (2019), 1215--1221. 

\bibitem{gh1}
H. Ghahramani, M.N. Ghosseiri, and L. Heidarizadeh, \textit{Linear maps on block upper triangular matrix algebras behaving like Jordan derivations through commutative zero products}, Oper. Matrices, 14 (2020). 189--205.

\bibitem{gh2}
Hoger Ghahramani, Mohammad Nader Ghosseiri, and Tahereh Rezaei, \textit{Characterizing Jordan derivable maps on triangular rings by local actions}, J. Mathematics, 2022 (2022), https://doi.org/10.1155/2022/9941760.

\bibitem{her}
I.N. Herstein, \textit{Jordan derivations on prime rings}, Proc. Amer. Math. Soc. 8 (1957), 1104--1110.

\bibitem{her2}
I.N. Herstein, \textit{Lie and Jordan structures in simple associative rings}, Bull. Amer. Math. Soc. 67 (1961), 517--531. 

\bibitem{joh1} 
B.E. Johnson, {\it Cohomology in Banach algebras}, Mem. Amer. Math. Soc. 127 (1972).

\bibitem{joh2}
B.E. Johnson. \textit{Approximate diagonals and cohomology of certain annihilator Banach algebras}, Amer. J. Math. 94 (1972), 685--698. 

\bibitem{joh3}
B. E. Johnson, \textit{Symmetric amenability and the nonexistence of Lie and Jordan derivations}, Math. Proc. Camb. Phil. Soc. 120 (1996), 455--473. 

\bibitem{kaniuth}
E. Kaniuth, A. Lau and J. Pym, \textit{On $\phi$-amenability of Banach algebras}, Math. Proc. Camp. Phil. Soc. 144 (2008), 85--96.

\bibitem{li}
Y. Li and F. Wei, \textit{Jordan derivations and Lie derivations on path algebras}, Bull. Iran. Math. Soc. 44 (2018), 79--92 . 

\bibitem{mar}
W.S. Martindale III, \textit{Lie derivations of primitive rings}, Michigan Math. J., 11 (1964), 183--187.


\bibitem{math}
M. Mathieu and A.R. Villena. \textit{The structure of Lie derivations on $C^{*}$-algebras}, J. Funct. Anal. 202 (2003), 504--525. 

\bibitem{mie}
C.R. Miers, \textit{Lie derivations of von Neumann algebras}, Duke Math. J., 40 (1973), 403--409.  

\bibitem{per}
A.M. Peralta and B. Russo, \textit{Automatic continuity of triple derivations on $C^{*}$-algebras and $JB^{*}$-triples}, J. Algebra, 399 (2014), 960--977. 

\bibitem{sahami}
A. Sahami and A. Pourabbas, \textit{On $\phi$-biflat and $\phi$-biprojective Banach algebras}, Bull. Belg. Math. Soc. Simon Stevin, 20(5) (2013), 789--801.

\bibitem{sangani}
M. Sangani Monfared, \textit{Character amenability of Banach algebras}, Math. Proc. Cambridge Philos. Soc. 144 (2008). 697--706.

\bibitem{sinc}
A.M. Sinclair, \textit{Jordan homomorphisms and derivations on semisimple Banach algebras}, Proc. Amer. Math. Soc. 24 (1970), 209--214. 






\end{thebibliography}

\end{document}